\documentclass[oneside,reqno,a4paper,11pt]{amsart}
\usepackage{hypernat}
\usepackage[colorlinks=true]{hyperref}

\newtheorem {theorem}{Theorem}[section]
\newtheorem {lemma}[theorem]{Lemma}
\newtheorem {corollary}[theorem]{Corollary}

\theoremstyle{remark}
\newtheorem {remark}{Remark}[section]

\theoremstyle{problem}

\newtheorem {definition}[theorem]{Definition}
\theoremstyle{plain} \numberwithin {equation}{section}

\everymath{\displaystyle}

\makeindex

\def\be{\begin{equation}}
\def\ee{\end{equation}}
\def\ba{\begin{aligned}}
\def\ea{\end{aligned}}

\def\g{\nabla}
\def\lap{\Delta}
\def\O{\Omega}

\def\R{\mathbb{R}}
\def\p{\partial}
\def\a{\alpha}

\def\f{\frac}
\def\th{\theta}

\def\p{\partial}
\def\div{\mathrm{div}}
\def\ld{\lambda}
\def\k{\kappa}
\def\spt{\mathrm{supp}}
\def\ra{\rightarrow}
\def\endproof{{\hfill$\Box$}\\}
\def\ch#1{\mathrm{cl}({#1})}
\def\dist{\mathrm{dist}}
\def\T{\mathbb{T}}

\begin{document}

\vspace{1cm}
\title{On blowup of classical solutions to the compressible Navier-Stokes equations}
\author{Zhouping XIN, Wei YAN}
\thanks{This research is supported in parts by Zheng Ge Ru Foundation, Hong Kong RGC Earmarked Research Grants
CUHK 4042/08P and CUHK 4041/11P.}

\address{Zhouping XIN\hfill\break\indent
Institute of Mathematical Science,\hfill\break\indent
The Chinese University of Hong Kong, Shatin, NT, Hong Kong}
\email{\href{zpxin@ims.cuhk.edu.hk}{zpxin@ims.cuhk.edu.hk}}

\address{Wei YAN\hfill\break\indent
Laboratory of Science and Technology on Computational Physics,\hfill\break\indent
Institute of Applied Mathematics and Computational Physics, Beijing, P.R.China}
\email{\href{wyanmath@gmail.com}{wyanmath@gmail.com}}
\maketitle

\begin{abstract}
In this paper, we study the finite time blow up of smooth solutions to the Compressible Navier-Stokes
system when the initial data contain vacuums. We prove that any classical solutions of viscous compressible fluids without heat conduction will blow
up in finite time, as long as the initial data
has an isolated mass group (see Definition \ref{img}). The results hold regardless of either the size of the initial data or the far fields being
vacuum or not. This improves the blowup results of Xin \cite{xin98} by removing the crucial
assumptions that the initial density has compact support and the smooth solution has finite total energy.
Furthermore, the analysis here also yields that any classical solutions of viscous compressible fluids  without
heat conduction in bounded domains or periodic domains will blow up in finite time, if the initial data
have an isolated mass group satisfying some suitable conditions.
\end{abstract}

\section{Introduction}
Consider the  following well-known compressible Navier-Stokes equations for viscous compressible fluids,
\be\label{ns}
\left\{\ba
&\f{\p\rho}{\p t} + \div(\rho u) = 0,\\
&\f{\p(\rho u)}{\p t} + \div(\rho u\otimes u) + \g p = \div T,\qquad (x,t)\in\O\times\R_+,\\
&\f{\p(\rho E)}{\p t}+\div(\rho E u + pu) = \div(uT) + \k \lap \th.
\ea\right.\ee
The initial data can be taken as
\be\label{init}
\rho(x,0) = \rho_0(x),\ u(x,0) = u_0(x),\ E(x,0) = E_0(x),\quad x\in\O,\ d\geq 2.
\ee
Here $\O\subseteq\R^d$ is a smooth domain in $\R^d$ or periodic domain $\T^d$, and $\rho, u, p, \th$ denote the density, velocity, pressure, internal
energy and temperature respectively.
The specific total energy $E=\f 12 |u|^2+e$, and $T$ is the stress tensor given by
\be
 T = \mu(\g u + \g u^t) + \ld (\div u)I.
\ee
$\mu$ and $ \ld$ are the coefficient of viscosity and second coefficient of viscosity, respectively.
$\k$ is the coefficient of heat conduction.
The pressure $p$ is determined by the equation of state
\be\label{eos}
p = R\rho \th,\quad p=(\gamma-1)\rho e,
\ee
where $R>0$ and $\gamma>1$ are constants.

In the case that the domain $\O$ has boundary, the standard no-slip boundary condition or
Navier-slip boundary condition will be supplemented.

In this paper, it will be always assumed that
\be\label{mul}
\mu >0,\quad \ld +\f 2d\mu > 0,\quad \k=0.
\ee

As one of the most important systems in continuum mechanics, the  theory of global well-posedness  of solutions to the Cauchy problem
and initial-boundary-value problem for the system (\ref{ns}) has been studied extensively in\cite{K3, K4, F1, Ch,
Hof1, Hof2, Hof3, Kaz, L2, MN, M1, Na, Se, xin98, hlx1, hlx2, hlx3, hx, LX} and the references therein.
In particular, non-vacuum small perturbations of a uniform non-vacuum constant state have been shown existing globally in time
 and remain smooth in any space dimensions \cite{Na, Se, MN, M1, Hof1}, while for general data which may contain vacuum states,
 only weak solutions are shown to exist for the isentropic compressible Navier-Stokes system in multi-dimension with special equation of state as in
 \cite {L2, F1}, yet the uniqueness and regularity of these weak solutions remain unknown. In contrary to the one-dimensional case \cite{Hof2}, it is
 still open whether vacuum states can form in finite time from non-vacuum initial and boundary data for the compressible Navier-Stokes systems in
 higher space dimensions. Despite
 the progress on various blow-up criterion \cite{hx, hlx1, hlx2}, and the surprising results on global well-posedness of the classical
 solution to the 3-dimensional compressible Isentropic Navier-Stokes system for initial data with small total energy but
 possible large oscillations and containing vacuum states \cite{hlx3}, the behavior near vacuum of solutions to the system
 (\ref{ns}) remains to be one of the central issues for the global well-posedness of smooth solutions to the general
 full compressible Navier-Stkes system (\ref{ns}), as is illustrated in the following result of blow-up of smooth
 solutions for the full compressible Navier-Stokes system (\ref{ns}) without heat conduction (i.e. satisfying (\ref{mul})):

\begin{theorem}\label{xin}
(from Theorem 1.3 in \cite{xin98})
Consider the compressible Navier-Stokes system (\ref{ns}) without heat-conduction, i.e. (\ref{mul}) is satisfied.
Then there is no non-trivial solution in $C^1([0,\infty), H^m(\R^d))$ to the Cauchy problem, (\ref{ns}) and (\ref{init}), provided that the initial
density has compact support.
\end{theorem}
This is the first result of finite time blow-up of smooth solutions for viscous compressible fluids proved by Xin in
\cite{xin98}, which was generalized by Cho and Jin \cite{Ch} to the case of $\k>0$.
 Later, Rozanova \cite{R} obtained a similar  blowup result under rapidly decay assumptions instead of compact support assumptions of initial data.
 Then, Luo and Xin \cite{LX} proved the finite time blowup of symmetric smooth solutions to two dimensional isentropic Navier-Stokes equations and analyzed
 the blowup behavior at infinity time for one point vacuum initial data. Recently, Du, Li and Zhang \cite{du11} show the blowup of smooth solutions to
 the isothermal case for one dimensional case and two dimensional case with spherically symmetric assumptions.

It should be noted that all the results mentioned above on the blowup of smooth solutions are for Cauchy problems
of the compressible Navier-Stokes equations, i.e., $\O=\R^d$, and there are two crucial assumptions that the density
 has compact support spatially (at least, the far field must be in vacuum state), and the velocity field must be in
 $C^1([0, \infty), H^m(\R^d))$ (which implies in particular that the solution has finite energy so that the velocity
 is well-defined even in the vacuum region). The first natural question is whether the far field being in vacuum
 (in particular, the density has compact support) is a necessary condition for the finite time blow-up of smooth solutions. Indeed, for isentropic
 compressible Navier-Stokes equations in two or three space dimensions, when the far
 fields are non-vacuum, there exist smooth global small energy solutions which may contain vacuum for the Cauchy problem, \cite{hlx3}. It is an open
 problem whether the similar theory holds for the full compressible Navier-Stokes system (\ref{ns}). Another important question is whether one can
 remove the assumption that $u\in H^m (\R^d)$ for suitable large $m$. It should be emphasized that this is a very strict assumption which plays a
 crucial role in the analysis of the blow-up results in \cite{xin98, Ch, LX, du11, R}. Yet, it is not clear physically why the velocity field has the
 asymptotic behavior at vacuum. Furthermore, even the local existence of smooth solutions for the Cauchy problems with initial data containing vacuum
 are proved only in the case that the velocity fields are in some suitable homogeneous space (in particular the velocity fields are not square
 integrable on $\R^d$, \cite{K3, K4}. It should be also noted that the global well-posedness of
 smooth solutions which may contain vacuum states and large oscillations for the Cauchy problem of the isentropic compressible Navier-Stokes in $\R^3$
 are in homogeneous spaces also \cite{hlx3}.
  Thus it is desirable to generalize the blow-up results
 in Theorem \ref{xin} to general classical solutions to the compressible Navier-Stokes system without the assumption that $u \in H^m(\R^d)$. Finally,
 since all the the previous blow-up results of smooth solutions concern only with Cauchy problems and it seems difficult to adapt the available
 methods to deal with the initial-boundary value problems and periodic problems which are also very important issues for compressible Navier-Stokes
 equations.

 In this paper, we will answer all three main questions mentioned above. First, we show  the finite time blow up of classical solutions to the Cauchy
problem for the compressible Navier-Stokes equations (\ref{ns}) without heat conduction
for a class of initial data containing vacuum but without any restrictions on the velocity fields at vacuum beyond
the regularity. The class of initial data includes the case that the initial density has compact support. The proof is
based on the key observation that if initially a positive mass is surrounded by a bounded vacuum region, then the time
evolution remains uniformly bounded for all time. Then this analysis can be modified easily to show the finite time blowup of classical solutions to
initial-boundary value problems and periodic problems under some suitable conditions.

The rest of this paper is organized as follows. In section \ref{notion}, we give some notions, state main results, and describe the main ideas of
the proof. Then the key estimates and the complete proofs of the results are given in section \ref{proof} and section \ref{proof3}.

\section{Notations and main theorems}\label{notion}

Before stating the main results, we introduce some notations. Recall that the classical solutions to the compressible
Navier-Stokes equations can be defined as follows:

\begin{definition}\label{cla} (Classical solutions) Let T be positive. A triple $(\rho(x,t), u(x,t), E(x,t))$ is called
a classical solution to the compressible Navier-Stokes system (\ref{ns}) on $\O \times (0,T)$ if $\rho \in C^1(\O \times [0,T)), (u,E) \in C^1([o,T),
C^2(\O))$, and satisfies the system (\ref{ns}) point-wisely on $\O \times (0,T)$. It is called a classical solution to the Cauchy problem (\ref{ns})
and (\ref{init}) if it is a classical solution to the system (\ref{ns}) on $\R^d \times (0,T)$ and takes on the initial data (\ref{init})
continuously. Similarly, it is called a
classical solution to the initial-boundary-value problem for the system (\ref{ns}) if it is a classical solution to the
system (\ref{ns}), takes the initial data (\ref{init}),  and satisfies the boundary conditions continuously.
\end{definition}

We now identify a class of initial data which contains vacuum states. Let the initial data for a classical solution
$(\rho(x,t), u(x,t), E(x,t))$ to the system (\ref{ns}) be defined in (\ref{init}).

\begin{definition}\label{img} (Non-periodic case) Suppose $\O$ be a smooth domain in $\R^d$. The pair $(V, U)$ is called an isolated mass group of
$\rho_0(x)$, if both $V\subset\O$ and $U\subset\O$ are bounded open sets, $U$ is connected, and satisfy
\be\label{h1}
\left\{\ba
&V\subset\overline V\subset U,\\
&\rho_0(x)=0,\quad \mathrm{ in }\ U - V,
\ea\right.
\ee
and $\rho_0(x)$ is not identically equal to zero on $V$. If $U\subseteq B_{R}(\bar x)$ for some $\bar x$, then $(V, U)$ is said to have radius $R$.
For simplicity,
$(V,U)$ is called an isolated mass group.
\end{definition}
In the periodic case, this definition is modified as follows:
\begin{definition}\label{img2}
(Periodic case) Suppose $\O=\T^d$. The pair $(V, U)$ is called an isolated mass group of $\rho_0(x)$, if
the pair $(V,U)$ is an isolated mass group of $\rho_0(x)$ in the sense of Definition \ref{img} after periodic extension $\T^d$ to $\R^d$.
If, after periodic extension, $U\subseteq B_{R}(\bar x)\subset \R^d$ for some $\bar x$,  $(V, U)$ is said to have radius $R$.
\end{definition}
\begin{remark}
It is noted that $B_R$ in Definition \ref{img2} may not be contained in $\T^d$.
\end{remark}
\begin{remark}
Since the Navier-Stokes equations are invariant under translation,
without loss generation, it will be assumed that $\bar x=0$ in this paper without explicit declaration.
\end{remark}

Let $(V, U)$ be an isolated mass group with radius $R_1$.
Denote by $m_0,\ x_0$ the initial mass and initial centroid of the isolated mass group $V$  respectively. That is,
\be\label{m0}
m_0 = \int_{V} \rho_0(x)dx >0,\quad x_0 = \f 1{m_0}\int_{V} \rho_0(x) xdx.
\ee
Set
\be\ba\label{m1234}
& m_1 = \int_{V}\rho_0(x) u_0(x) dx, && m_2 = \int_{V}|x|^2\rho_0(x) dx>0,\\
& m_3=\int_{V}\rho_0(x) (u_0(x)-\f{m_1}{m_0})\cdot xdx,&&m_4 = \int_{V}\rho_0(x) (\f 12 |u_0(x)-\f{m_1}{m_0}|^2+e_0(x))dx>0.
\ea\ee
Denote by $T^*$  the only positive root of
\be\label{ts}
\min(2, d(\gamma-1)) m_4 t^2  + 2m_3 t -( (2R_1 + |x_0|)^2 m_0 - m_2)= 0.
\ee
Let $\ch{U}$ be the closed convex hull of $U$.

Then the main results in this paper are the following three theorems. The first theorem concerns with the Cauchy problem.
\begin{theorem}\label{main}
Consider the full  compressible Navier-Stokes system (\ref{ns})   without heat-conduction, i.e., satisfying (\ref{mul}).
Assume that initial density has an isolated mass group.
Then there is no global in time classical  solution  to the Cauchy problem for the Navier-Stokes system (\ref{ns}) and (\ref{init}).
\end{theorem}
In the case of periodic domains, we have:
\begin{theorem}\label{main2}
Consider the viscous compressible flows without heat-conduction in the periodic domain $\T^d$.
Suppose that the viscosity coefficients $\mu,\ \ld$ satisfy (\ref{mul})
and the initial density has an isolated mass group $(V, U)$.
Then there is no global in time classical  solution  to the Cauchy problem for the compressible Navier-Stokes system (\ref{ns}) with initial data (\ref{init}).
\end{theorem}
Finally,  we deal with the initial boundary value problem.
\begin{theorem}\label{main3}
Consider the  viscous compressible flows without heat-conduction in a  smooth domain $\O\subset\R^d$.
Suppose that the viscocity coefficients $\mu,\ \ld$ satisfy (\ref{mul}).
and the initial density has an isolated mass group $(V, U)$ with radius $R_1$.
Assume further  that
\be\label{e11}
\ba
&B_{2R_1 + |x_0| + \f{|m_1|}{m_0}T^*}(0)\subset \O,
\ea\ee
where $T^*$ is defined in (\ref{ts}).
Then there is no global in time classical solution to the initial-boundary value  problem for the compressible Navier-Stokes system (\ref{ns})
with initial data  (\ref{init}) and suitable boundary conditions.
\end{theorem}

One of the interesting corollaries of Theorem \ref{main} is the following strong version of the blow-up results of Xin  \cite{xin98}:
\begin{corollary}\label{cor1}
Assume that the viscosity coefficients satisfy the condition (\ref{mul}). Then there is no non-trivial global in time classical solution to the
Cauchy problem for the compressible Navier-Stokes system (\ref{ns}) with initial data (\ref{init}), provided that the initial density has compact
support, i.e.,
\be
\spt\rho_0(x)\subset\subset B_{R_0}(0).
\ee
\end{corollary}

\begin{remark}
Theorem \ref{main3} holds independent of the boundary conditions on $\p\O$.
\end{remark}

\begin{remark}
Theorem \ref{main} shows that
any classical solution to the  compressible Navier-Stokes system  without
heat conduction will blow up in finite time, as long as its initial data
has an isolated mass group,
no matter how small the initial data is and no matter whether the states of its far fields are vacuum or not.
\end{remark}

\begin{remark}
It should be noted that the condition (\ref{e11}) can be loosen slightly to
\be\label{e12}
\ba
&\ch{B_{2R_1 + |x_0|}(0)\cup B_{2R_1+|x_0|}(\f{m_1}{m_0}T^*)} \subset \O.
\ea\ee
These conditions guarantee that the boundaries of isolated mass groups are  away from the boundary of the domain
in the lifespan of the classical solution.
\end{remark}

Now, we make some comments on  the main ideas of the proofs.
Recall that there are two crucial elements in the proof of the blowup result, Theorem \ref{xin}, in \cite{xin98}. The first is that the total
pressure over $\R^d$  decays fast
in time, and the second  is the support of the density grows sub-linearly in time. It is noted that the decay of the total pressure involves some
integrability of powers
of the density,  which can be guaranteed by the condition that the density has compact support,  but  has nothing to do with the integrability of the
velocity field
over $\R^d$. However, the  assumption $u\in H^m$ plays an essential role to show that $u(x)=0$ in
(unbounded) vacuum
regions which ensures that the support of density is preserved in time, and thus contradicts to the fact that
the second moment increases in the order at least as $t^2$  as $t\ra\infty$ \cite{xin98, Ch, du11, LX}. To deal with the general case here, we  note
that the argument  in
 \cite{xin98} show that, for classical  solutions to the compressible Navier-Stokes equations (\ref{ns}) without heat conduction, it holds that  in
 vacuum regions,
\be\label{Tu}
\left\{\ba
& \div T=0,\\
& \div(uT) = 0.
\ea\right.
\ee
It follows from (\ref{Tu}) and (\ref{mul}) that in vacuum regions,
\be\label{x1}
\left\{\ba
&\p_i u_i = 0,\quad 1\leq i\leq d,\\
&\p_i u_j + \p_j u_i = 0,\quad i\neq j.\\
\ea\right.
\ee
For the special case that density has compact support, (\ref{x1}), together with the assumption $u\in H^m$,
implies that $u=0$ in unbounded vacuum regions. For the general case, one of the key observations  in this paper is that  (\ref{x1}) implies that in
the vacuum regions,
\be\label{ee}
u(x,t) = A(t)x + b(t),
\ee
where $A(t)$ is a antisymmetric matrix. Thus, in general,  $u(x,t)$ may not be zero and may not even be integrable in vacuum regions and
the interfaces between the fluids and the  vacuum states vary in time. Thus the previous analysis cannot work in general. In this paper,  instead
studying
the evolution of the initial states over $\R^d$,  we consider  only the dynamic motion of an isolated mass group. It will be shown that  (\ref{Tu})
and
(\ref{ee}) guarantee that the mass, momentum, energy, and centroid of  the time evolution of the initial  isolated mass group are time invariant,
and the diameter
of the time evolution of the initial isolated mass group remains uniformly bounded up to a translation.
Intuitively, the $A(t)x$ term of $u(x,t)$ expresses the motions in the direction perpendicular to $x$ and so does not increase the bounds on the
diameter of
the initial  isolated mass groups. While the term of $b(t)$ in $u(x,t)$ means that the isolated mass group may translate.
The translations of such isolated mass group can be controlled by the motion of its  centroid and thus
be controlled by the  momentums. But the  momentums of isolated mass group are conserved due to the Navier-Stokes euations (\ref{ns})
and the behaviors of classical  solutions in vacuum regions.

\section{Proof of Theorem \ref{main}, Theorem \ref{main2} and Corollary \ref{cor1}.}\label{proof}
In this section, we prove Theorem \ref{main}, Theorem \ref{main2} and Corollary \ref{cor1}. Let $(\rho,u,E)(x,t)$ be a classical solution to the
Navier-Stokes
equation (\ref{ns}) without heat conduction and $(V,U)$ be an isolated mass group for the corresponding initial density. Without loss of generality,
it will be assumed that  the
initial total momentum of $V$ is zero, i.e.
\be\label{h2}
m_1 = 0.
\ee
Otherwise, one can  take the following Galilean transformation to achieve this:
\[\left\{\ba
& t' = t,\quad
 x' = x + \f{m_1}{m_0}t,\\
& \rho'(x',t') = \rho(x,t),\ u'(x',t') = u(x, t) - \f{m_1}{m_0},\ e'(x',t')=e(x,t).
\ea\right.\]

Assume that the isolated mass group $(V, U)$ has a radius of $R_1$, i.e.
\[
U\subset B_{R_1}(0).
\]
Let $X(\a, t)$ be the particle path starting at $\a$ when $t=0$,
\[
\left\{\ba
&\f{d}{dt}X(\a, t) = u(X(\a,t),t),\\
&X(\a,0) = \a.
\ea\right.
\]
Define
\be\left\{\ba
&\O_1(t) = \{ X(\a,t)\big|\ \a\in V\},\\
&\O_2(t) = \{ X(\a,t)\big|\ \a\in U\},\\
&\O_0(t) = \{ X(\a,t)\big|\ \a\in U-V\}.\\
\ea\right.
\ee
Then,
\[
\O_2(t) = \O_1(t)\cup\O_0(t).
\]

Since $\rho_0(x) = 0$ in $U-V$,  it follows from  the mass equation that
\[
\rho(x,t) = 0,\quad \mathrm{in}\ \O_0(t).
\]
Then our first observation is the form of the velocity field in the vacuum region $\O_0(t)$ as follows.
\begin{lemma}\label{vu}
There exist an antisymmetric matrix $A(t)$ and a vector $b(t)$ such that
\be\label{u}
u(x,t) = A(t) x + b(t),\quad x\in\O_0(t).
\ee
Moreover,
\be\label{T0}
T(x,t) = 0, \quad x\in\O_0(t),
\ee
and
\be\label{Tx}
\int_{\O_2(t)}(\div T)\cdot x dx = 0.
\ee
\end{lemma}
\begin{proof}
Under the condition (\ref{mul}), it follows from the arguments of Xin  in \cite{xin98} that
\be\label{du}
\g u + \g^t u(x,t) =0, \quad x\in\O_0(t),
\ee
or
\[
\p_i u_j + \p_j u_i=0,\quad 1\leq i,j\leq d,\ x\in\O_0(t).
\]

Then, for any $1\leq i,j,k\leq d$,
\[
\p^2_{ij}u_k = \p_i(\p_j u_k) = -\p_i(\p_k u_j) = -\p^2_{ik} u_j.
\]
On the other hand,
\[
\p^2_{ij}u_k = \p_j(\p_i u_k) = -\p_j(\p_k u_i) = - \p_k(\p_j u_i) = \p_k(\p_i u_j) = \p^2_{ik} u_j.
\]
Therefore,
\[
\p^2_{ij} u_k = 0,\quad 1\leq i,j,k\leq d,\ x\in\O_0(t).
\]
This yields that there exist a matrix $A(t)$ and a vector $b(t)$ such that
\be\label{e01}
u(x,t) = A(t) x + b(t),\ x\in\O_0(t).
\ee
Substituting (\ref{e01}) into (\ref{du}) gives
\[
A(t) + A^t(t) = 0.
\]
So $A(t)$ is antisymmetric. This in turn implies (\ref{T0}) trivially.

Now we turn to prove (\ref{Tx}). Direct calculations show that
\[
\ba
\int_{\O_2(t)} &\div(T)\cdot x dx = -\int_{\O_2(t)}\mathrm{trac}(T) dx
= -(2\mu + d\ld)\int_{\O_2(t)} \div u(x,t)dx\\
&= -(2\mu + d\ld)\int_{\p\O_2(t)} u(x,t)\cdot n ds
=-(2\mu + d\ld)\int_{\p\O_2(t)} (A(t) x + b(t))\cdot nds\\
&=-(2\mu + d\ld)\int_{\O_2(t)} \div (A(t) x+ b(t)) dx
=-(2\mu + d\ld)\int_{\O_2(t)} \mathrm{trac}(A(t)) dx\\
&=0.
\ea
\]
\end{proof}
Based on this lemma, the diameter of the time evolution of the isolated mass group, i.e.,
the diameter of $\O_2(t)$, can be estimated. To this end, we define
 $D(t)$ to be a smooth vector function of $t$ satisfying
\be
\left\{\ba
&\f{d D(t)}{d t} = A(t) D(t) + b(t),\\
&D(0) = 0.
\ea\right.
\ee
Since $A(t)$ and $b(t)$ are smooth, $D(t)$ is well-defined by the classical theory of ordinary differential equations.
Then we can get
\begin{lemma}\label{lem2} It holds that
\be
|X(\a, t) - D(t)| = |\a|\leq R_1,\quad\text{ for any }\a\in U-V,
\ee
and
\be\label{um}
|x - D(t)|\leq R_1, \quad x\in\O_0(t).
\ee
and
\be\label{ch2}
\O_2(t)\subset B_{R_1}(D(t)),\quad\ch{\O_2(t)}\subseteq \overline{B_{R_1}(D(t))}.
\ee
\end{lemma}
\begin{proof}
Let $Y(\a, t) = X(\a, t) - D(t)$. Direct calculations yield that
\[
\f{dY(\a, t)}{dt} = A(t) Y(\a, t),\quad Y(\a, 0) = \a,\quad \a\in U-V.
\]
Since $A(t)$ is antisymmetric,
\[
\f12 \f{d(|Y(\a, t)|^2)}{dt} = Y(\a, t)\cdot(A(t) Y(\a, t)) = 0,\quad \a\in U-V.
\]
Therefore,
\[
|X(\a, t) - D(t)| = |Y(\a, t)| = |Y(\a, 0)| = |\a|,\quad \forall\ x\in U-V.
\]
Consequently, (\ref{um}) and then (\ref{ch2}) follow  trivially by a simple topological argument.
\end{proof}

To study the conservation laws for the evolution of the isolated mass group, one needs the following elementary transportation formula.
\begin{lemma}\label{lem3}
For any $F(x,t)\in C^1(\R^d\times \R_+)$,
\[
\f {d}{dt}\int_{\O_2(t)} F(x,t) dx = \int_{\O_2(t)}\p_t F(x,t) dx + \int_{\p\O_2(t)} F(x,t)( u(x,t)\cdot n) ds.
\]
\end{lemma}
\begin{proof} The proof is a  simple calculation, and so is omitted.
\end{proof}
Then we have the following conserved quantities for the time evolution of the isolated mass group.
\begin{lemma}\label{cons}
The total mass, total momentum, total energy, and centroid of $\O_2(t)$ are conserved, i.e.
\be\label{mas}
\int_{\O_2(t)}\rho(x,t) dx = m_0,
\ee
\be\label{mom}
\int_{\O_2(t)}\rho(x,t) u(x,t) dx = m_1 =0,
\ee
\be\label{eng}
\int_{\O_2(t)}\rho E dx = m_4,
\ee
and
\be\label{cent}
\f{\int_{\O_2(t)}\rho(x,t) xdx}{\int_{\O_2(t)}\rho(x,t) dx} = x_0.
\ee
\end{lemma}
\begin{proof}
Integrating the mass equation, momentum equation and energy equation on $\O_2(t)$ respectively and using  Lemma \ref{lem3} and Lemma \ref{vu},
one can derive the conservation of the total mass, (\ref{mas}), total momentum, (\ref{mom}), and total energy, (\ref{eng}), over $\O_2(t)$ easily. The
invariance of the centroid, (\ref{cent}),  follows from the calculations in the proof  of the next lemma.
\end{proof}

Next, we estimate the sizes of $D(t)$ and $\O_1(t)$.
\begin{lemma}\label{ct} {(\bf Key estimates)} It holds that
\be\label{dt}
|D(t)|\leq R_1 + |x_0|,
\ee
and
\be\label{spt}
\O_1(t)\subset \O_2(t)\subseteq B_{2R_1 + |x_0|}.
\ee

\end{lemma}\label{lem1}
\begin{proof}
Multiplying the  mass equation by $x$, integrating on $\O_2(t)$ and using Lemma \ref{lem3} and Lemma \ref{cons}, one gets that
\be\label{cm1}
\ba
\f{d}{dt}\int_{\O_2(t)}\rho(x,t) x dx &= -\int_{\O_2(t)}x \div(\rho u)dx = -\int_{\O_2(t)}\div(\rho u\otimes x)dx + \int_{\O_2(t)}\rho udx\\
&=\int_{\O_2(t)}\rho udx = m_1 = 0.
\ea
\ee
where
\[
\int_{\p\O_2(t)} x\rho(x,t)(u(x,t)\cdot n) ds = 0.
\]
(\ref{cm1}) yields that
\be
\int_{\O_2(t)}\rho(x,t) x dx =\int_{U}\rho_0(x) x dx = \int_{V}\rho_0(x) x dx =m_0 x_0.
\ee
Since the mass over $\O_2(t)$ is conserved, thus $x_0$ is contained in the closed convex hull of $\O_2(t)$, that is,
\be\label{e2}
x_0 \in \ch{\O_2(t)}\subseteq \overline{B_{R_1}(D(t))}.
\ee

(\ref{e2}) implies
\[
|x_0-D(t)|\leq R_1.
\]
This yields the estimate (\ref{dt}). (\ref{spt}) is a direct consequence of (\ref{ch2}) and (\ref{dt}).
\end{proof}

With the estimate  (\ref{spt}) on the size  of the isolated mass group $(\O_1(t), \O_2(t)$  at hand, Theorem \ref{main} can be proved by using a similar
argument in \cite{xin98} (see also \cite{Ch}) with some proper modifications. For completeness, we will give the details here.

{\it Proof of Theorem \ref{main}.}
Multiplying the mass equation by $|x|^2$  and using Lemma \ref{lem3}  show
\[
\f{d}{dt}\int_{\O_2(t)}\rho |x|^2dx = 2\int_{\O_2(t)} \rho u\cdot xdx,
\]
and
\be\label{e3}
\int_{\O_2(t)}\rho |x|^2dx = m_2 + 2\int_0^t\int_{\O_2(s)} \rho u\cdot x dx ds.
\ee

Taking inner product of the momentum equations with  $x$, integrating by part, and using (\ref{Tx})  lead to
\[
\f{d}{dt}\int_{\O_2(t)}\rho u\cdot xdx = \int_{\O_2(t)}\rho |u|^2dx + d\int_{\O_2(t)} pdx.
\]
Then
\be\label{e4}
\int_{\O_2(t)}\rho u\cdot xdx = m_3 + \int_0^t\int_{\O_2(s)}\rho |u|^2 dxds + d\int_0^t\int_{\O_2(s)} pdxds.
\ee

It follows from  (\ref{e4}), (\ref{eng}),  and the equations of state (\ref{eos})  that
\be\label{e02}
\ba
\int_{\O_2(t)}\rho u\cdot x dx &= m_3 + 2\int_0^t\int_{\O_2(s)}\rho(x,s) E(x,s) dx ds +
(d-\f{2}{\gamma-1})\int_0^t\int_{\O_2(s)}p(x,s)ds\\
&=m_3 + 2m_4 t + (d(\gamma-1)-2)\int_0^t\int_{\O_2(s)}\rho(x,s) e(x,s) dxds.
\ea
\ee

If $(d(\gamma-1)-2)\geq 0$, (\ref{e02}) gives
\[
\int_{\O_2(t)}\rho u\cdot x dx\geq m_3 + 2m_4 t.
\]
If $(d(\gamma-1)-2)\leq 0$, combining $\int_{\O_2(t)}\rho e dx\leq \int_{\O_2(t)}\rho E dx=m_4$ with (\ref{e02}) yields that
\[\ba
\int_{\O_2(t)}\rho u\cdot x dx &\geq m_3 + 2m_4 t  - (2-d(\gamma-1))\int_0^t\int_{\O_2(s)}\rho e dxds\\
&\geq m_3 + 2m_4 t - (2-d(\gamma-1))m_4 t\\
& = m_3 + d(\gamma-1)m_4 t.
\ea
\]
Therefore, it always holds that
\be\label{e8}
\int_{\O_2(t)}\rho u\cdot x dx\geq m_3 + \min(2, d(\gamma-1))m_4 t.
\ee

Substituting (\ref{e8}) into (\ref{e3}) shows that
\be\label{e6}
\int_{\O_2(t)}\rho |x|^2dx \geq m_2 + 2m_3 t + \min(2, d(\gamma-1)) m_4 t^2.
\ee

On the other hand, it follows from (\ref{spt}) that
\be\label{e7}
\int_{\O_2(t)}\rho |x|^2 dx \leq R_2^2\int_{\O_2(t)}\rho dx = R_2^2 m_0.
\ee
with $R_2 = 2R_1 + |x_0|$.

Combining (\ref{e6}) with (\ref{e7}) yields that
\[
 m_2 + 2m_3 t + \min(2, d(\gamma-1)) m_4 t^2\leq R_2^2 m_0.
\]
This leads to a finite bound on the life span of the classical solution. Thus the theorem is proved.
\endproof

{\it Proof of Theorem \ref{main2}.}
This follows by taking periodic extension and then using Theorem \ref{main}.
\endproof

{\it Proof of Corollary \ref{cor1}. }
Taking $V = B_{R_0}$ and $U = B_{2R_0}$ in Theorem \ref{main} yields the desired conclusion.
\endproof

\section{Proof of Theorem \ref{main3}}\label{proof3}

In this section, we will prove Theorem \ref{main3}.
Note that the key procedure in the proof of Theorem \ref{main} is the analysis of the time evolution of the
isolated mass group. It follows clearly from the analysis in the previous section that to prove Theorem \ref{main3},  it suffices
 to show that, under the given conditions, the time evolution of the isolated mass group $(V, U)$
does not touch the boundary of the physical domain.

{\it Proof of Theorem \ref{main3}.} We will use the same notations as the previous section. Consider the case  $m_1 = 0$ first.
Assume a priori that, for any $t$ in the lifespan $[0, T)$ of the classical solution,
\be\label{e21}
\dist (\p\O_2(t),\ \p\O) >0\ \mathrm{or}\ \O_2(t)\subset\subset\O,\quad 0\leq t < T.
\ee
Then,  the same analysis as in section \ref{proof} shows that
\be\label{e22}
\ch{\O_2(t)} \subset B_{2R_1 + |x_0|}(0),
\ee
and
\[
 m_2 + 2m_3 t + \min(2, d(\gamma-1)) m_4 t^2\leq (2R_1 + |x_0|)^2 m_0.
\]
Therefore,
\be\label{e23}
 T\leq T^*.
\ee
Since (\ref{e22}), (\ref{e23}) and (\ref{e11}) (or (\ref{e12})) guarantee (\ref{e21}), so Theorem \ref{main3} holds for the case $m_1=0$.

Now, consider the case  $m_1\neq 0$. Take the following Galilean transformation,
\[\left\{\ba
& t' = t,\quad
 x' = x + \f{m_1}{m_0}t,\\
& \rho'(x',t') = \rho(x,t),\ u'(x',t') = u(x, t) - \f{m_1}{m_0},\ e'(x',t')=e(x,t).
\ea\right.\]
Since (\ref{m0}), (\ref{m1234}) and (\ref{ts}) are invariant under above Galilean transformation,
then the results can proved just as the case  $m_1=0$, except that the boundary of the domain, $\p\O(t)$,  becomes unsteady with the constant velocity
$\f{m_1}{m_0}$.

In this case, the condition (\ref{e21}) becomes
\[
\O_2(t)\subset\subset\O + \f{m_1}{m_0}t,\quad 0\leq t < T.
\]
However, it should be  clear that this indeed can be deduced from the assumption (\ref{e11}) or (\ref{e12}). So the proof Theorem \ref{main3} is
completed.
\endproof

\begin{remark}
Previous analysis also applies to the full Navier-Stokes equations with positive heat conduction under the assumption that the specific entropy is finite in vacuum regions (e.g. the case in \cite{Ch}). We omit the details here.
\end{remark}

\begin{remark}
It is worth noting that, recently, Huang and Li \cite{hl11} proved the global existence of the classical solutions with small energy to the full Compressible Navier-Stokes equations with positive heat conduction in the whole space $\R^d$, if the initial data tend to a constant non-vacuum state
in far fields and initial vacuum is allowed. However, the global classical solutions obtained in \cite{hl11} must  have positive absolute temperature $\th$ in vacuum regions (thus the specific entropy is infinity at vacuum regions).
\end{remark}

\end{document}